\documentclass[amssymb,reqno,amsfonts,refcheck,12pt,verbatim,righttag]{amsart}

\usepackage{graphicx}
\usepackage{graphics,color}
\usepackage{amssymb,mathrsfs}
\usepackage{graphicx,color,tikz,caption,subcaption}
\usepackage[all]{xy}
\usepackage{mathtools}
\usepackage{enumerate}
\usepackage{amsmath}
\usepackage{bbm}

\usepackage[colorlinks, linkcolor=blue,citecolor=blue, anchorcolor=blue,urlcolor=blue,pagebackref,hypertexnames=false]{hyperref}

\numberwithin{equation}{section} 
\newtheorem{thm}{Theorem}[section]
\newtheorem{lem}[thm]{Lemma}

\newtheorem{de}[thm]{Definition}

\newcommand{\R}{\mathbb{R}}

\newcommand{\N}{\mathbb{N}}

\usepackage{float}

\begin{document}
\baselineskip 14pt
\title{Inhomogeneous and simultaneous Diophantine approximation in beta dynamical systems}

\author{Yu-Feng Wu}
\address[]{School of Mathematics and Statistics\\ Central South University\\ Changsha, 410085, PR China}
\email{yufengwu.wu@gmail.com}

\keywords{beta dynamical systems,  Diophantine approximation, Hausdorff dimension}
\thanks{2010 {\it Mathematics Subject Classification}: 11K55, 28A80}


\date{}

\begin{abstract}
In this paper, we investigate inhomogeneous and simultaneous Diophantine approximation in beta dynamical systems. For $\beta>1$ let $T_{\beta}$ be the $\beta$-transformation on $[0,1]$. We determine the Lebesgue measure and Hausdorff dimension of the set 
\[\left\{(x,y)\in [0,1]^2: |T_{\beta}^nx-f(x,y)|<\varphi(n)\text{ for infinitely many }n\in\N\right\},\]
where $f:[0,1]^2\to [0,1]$ is a Lipschitz function and $\varphi$ is a positive function on $\N$. Let $\beta_2\geq \beta_1>1$, $f_1,f_2:[0,1]\to [0,1]$ be two Lipschitz functions, $\tau_1,\tau_2$ be two positive continuous functions on $[0,1]$. We also determine the Hausdorff dimension of the set 
\[\left\{(x,y)\in [0,1]^2: \begin{aligned}&|T_{\beta_1}^nx-f_1(x)|<\beta_1^{-n\tau_1(x)}\\ &|T_{\beta_2}^ny-f_2(y)|<\beta_2^{-n\tau_2(y)}\end{aligned}\text{ for infinitely many }n\in\N\right\}.\]
Under certain additional assumptions,  the Hausdorff dimension of the set 
\[\left\{(x,y)\in [0,1]^2: \begin{aligned}&|T_{\beta_1}^nx-g_1(x,y)|<\beta_1^{-n\tau_1(x)}\\ &|T_{\beta_2}^ny-g_2(x,y)|<\beta_2^{-n\tau_2(y)}\end{aligned}\text{ for infinitely many }n\in\N\right\}\]
is also determined, where $g_1,g_2:[0,1]^2\to [0,1]$ are two Lipschitz functions.
\end{abstract}

\maketitle

\section{Introduction}\label{S-1}

Given a real number $\beta>1$, the $\beta$-transformation $T_{\beta}: [0,1]\to [0,1]$ is define by 
\[T_{\beta}x=\beta x-\lfloor \beta x\rfloor \quad \text{ for }x\in [0,1],\]
where $\lfloor \cdot \rfloor$ denotes the integral part of a real number. In this paper, we study the metric properties of orbits in the beta dynamical system $([0,1], T_{\beta})$. 

In 1957, R{\'en}yi \cite{Renyi57} initialed the study of the map $T_{\beta}$ in his investigation of expansions of real numbers in non-integral bases. Parry \cite{Parry60} proved that there is an invariant and ergodic measure for $T_{\beta}$, which is equivalent to the Lebesgue measure $\mathcal{L}$ on $[0,1]$. Then  Birkhoff's ergodic theorem implies that for any fixed $x_0\in [0,1]$, 
\begin{equation}\label{eqbirkhoff}
\liminf_{n\to\infty}|T_{\beta}^nx-x_0|=0
\end{equation}
for $\mathcal{L}$-almost all $x\in [0,1]$. On the other hand, the Poincar{\'e} Recurrence Theorem yields that for $\mathcal{L}$-almost all $x\in [0,1]$, 
\begin{equation}\label{eqPoincare}
\liminf_{n\to\infty}|T_{\beta}^nx-x|=0.
\end{equation}
Both \eqref{eqbirkhoff} and \eqref{eqPoincare}  are qualitative in nature, and taking into account the speed of convergence in \eqref{eqbirkhoff} and \eqref{eqPoincare} leads to the study of the metric properties of the set
\begin{equation}\label{eqDbetax0}
D_{\beta}(\varphi,x_0)=\left\{x\in [0,1]: |T_{\beta}^nx-x_0|<\varphi(n)\text{ for infinitely many }n\in\N\right\}
\end{equation}
and the set 
\begin{equation}\label{eqRvarphix}
R_{\beta}(\varphi)=\left\{x\in [0,1]: |T_{\beta}^nx-x|<\varphi(n)\text{ for infinitely many }n\in\N\right\},
\end{equation}
respectively, where $\varphi:\N\to(0,\infty)$ is a positive function. 
The study of $D_{\beta}(\varphi,x_0)$ is called the {\em shrinking target problem}, and the study of $R_{\beta}(\varphi)$ is called {\em quantitative recurrence} in beta dynamical systems. When $x_0\neq 0$, the set $D_{\beta}(\varphi,x_0)$ can be also viewed as the {\em inhomogeneous Diophantine approximation} by orbits in $([0,1], T_{\beta})$.  In the following, we introduce related works which motivated the present paper. In a general measure preserving dynamical system with compatible metric, the shrinking target problem was introduced by Hill and Velani \cite{HillVelani95}, and for a pioneering work on quantitative recurrence one refers to Boshernitzan \cite{Boshernitzan93}. 

In 1967, Philipp \cite{PhilippW67} proved that for any fixed $x_0\in [0,1]$,
\begin{equation}\label{eq01law}
\mathcal{L}(D_{\beta}(\varphi,x_0))=\begin{cases}0&\text{ if }\sum_{n=1}^{\infty}\varphi(n)<\infty,\\
1&\text{ if }\sum_{n=1}^{\infty}\varphi(n)=\infty.
\end{cases}
\end{equation}
When $\sum_{n=1}^{\infty}\varphi(n)<\infty$, Shen and Wang \cite{ShenWang13} obtained the Hausdorff dimension of $D_{\beta}(\varphi,x_0)$. They proved that 
\begin{equation}\label{eqdimHbetavx0}
\dim_{\rm H}D_{\beta}(\varphi, x_0)=\frac{1}{1+\alpha} \quad \text{ with }\alpha=\liminf_{n\to\infty}\frac{\log_{\beta}\varphi(n)^{-1}}{n},
\end{equation}
where $\dim_{\rm H}$ denotes the Hausdorff dimension. 

Concerning the set $R_{\beta}(\varphi)$ in \eqref{eqRvarphix}, Tan and Wang \cite{TanWang11} obtained  its Hausdorff dimension, which is also equal to $\frac{1}{1+\alpha}$ given in \eqref{eqdimHbetavx0}. Wang \cite{WangW18} considered the following extension of $R_{\beta}(\varphi)$:
\begin{equation}\label{eqsetxTfx}
\left\{x\in [0,1]: |T_{\beta}^nx-f(x)|<\varphi(n)\text{ for infinitely many }n\in\N\right\},
\end{equation}
where $f:[0,1]\to [0,1]$ is a Lipschitz function. He proved that this set has  Hausdorff dimension $\frac{1}{1+\alpha}$ given in \eqref{eqdimHbetavx0} as well. The set \eqref{eqsetxTfx} can be viewed as the inhomogeneous Diophantine approximation by orbits in $([0,1], T_{\beta})$ in which the inhomogeneous part (i.e. the term $f(x)$) is allowed to vary.  Very recently, L{\"u}, Wang and Wu \cite{FLBWJW21} proved that the  Lebesgue measure of the set \eqref{eqsetxTfx} also satisfies  \eqref{eq01law}. 

Another related set is the set of points $(x,y)\in [0,1]^2$ such that the orbit of $x$ under $T_{\beta}$ approximates $y$ with a given speed, which is studied by Ge and L{\"u} \cite{GYHLF15}. More precisely, Ge and L{\"u} proved that the two-dimensional Lebesgue measure $\mathcal{L}^2$ of the set \
\begin{equation}\label{eqGeLuset}
\left\{(x,y)\in [0,1]^2: |T_{\beta}^nx-y|<\varphi(n)\text{ for infinitely many }n\in\N\right\}
\end{equation}
satisfies \eqref{eq01law} and this set has  Hausdorff dimension $1+\frac{1}{1+\alpha}$ with $\alpha$ given in \eqref{eqdimHbetavx0}. Coons, Hussain and Wang \cite{CHW16} extended this result to the  generalised Hausdorff measure.

In view of the sets \eqref{eqDbetax0}, \eqref{eqRvarphix}, \eqref{eqsetxTfx} and \eqref{eqGeLuset}, and the corresponding results mentioned above, a natural question is what if we replace the inhomogeneous part $y$ in  \eqref{eqGeLuset} by a more general function, especially depending on both coordinates. Concerning this question, we have the following result. 

\begin{thm}\label{thm1}
Let $\beta>1$, $f:[0,1]^2\to [0,1]$ be a Lipschitz function, and $\varphi$ be a positive function on $\N$. Set
\[W_{\beta}(f,\varphi)=\left\{(x,y)\in [0,1]^2: |T_{\beta}^nx-f(x,y)|<\varphi(n)\text{ for infinitely many }n\in\N\right\}.\] Then 
\[\mathcal{L}^2(W_{\beta}(f, \varphi))=\begin{cases} 0 & \text{ if }\sum_{n=1}^{\infty}\varphi(n)<\infty,\\
1 & \text{ if }\sum_{n=1}^{\infty}\varphi(n)=\infty.
\end{cases}\]
Furthermore, if $\sum_{n=1}^{\infty}\varphi(n)<\infty$, then 
\[\dim_{\rm H}W_{\beta}(f, \varphi)=1+\frac{1}{1+\alpha},\]
where $\alpha=\liminf_{n\to\infty}\frac{\log_{\beta}\varphi(n)^{-1}}{n}$. 
\end{thm}

This paper was also motivated by a result of Wang and Li \cite{WWLLL20} on {\em simultaneous Diophantine approximation} in beta dynamical systems. More precisely, let $f, g: [0,1]\to [0,1]$ be two Lipschitz functions, $\tau_1,\tau_2$ be two  positive continuous functions on $[0,1]$ with $\tau_1(x)\leq \tau_2(y)$ for all $x,y\in [0,1]$. Wang and Li \cite{WWLLL20} proved that for any  $\beta>1$,  the Hausdorff dimension of the set 
\begin{equation*}
\left\{(x,y)\in [0,1]^2: \begin{aligned}&|T_{\beta}^nx-f(x)|<\beta^{-n\tau_1(x)}\\ &|T_{\beta}^ny-g(y)|<\beta^{-n\tau_2(y)}\end{aligned}\text{ for infinitely many }n\in\N\right\}
\end{equation*}
is equal to 
\[\min\left\{\frac{2}{1+\theta_{1}}, \frac{2+\theta_{2}-\theta_{1}}{1+\theta_{2}}\right\},\]
where $\theta_{i}=\min_{x\in [0,1]}\tau_i(x)$ for $i=1,2$. 
This generalizes a result of  Hussain and Wang \cite{HW18}, in which they obtained the Hausdorff dimension of the above set when $f,g,\tau_1,\tau_2$ are all constants. 
It is  natural to consider approximation by orbits under two (maybe different) transformations $T_{\beta_1}$ and $T_{\beta_2}$ in the $x$- and $y$-coordinate, respectively. Concerning this, we have the following result.

\begin{thm}\label{thm2}
Let $\beta_2\geq \beta_1>1$, $f_1,f_2:[0,1]\to[0,1]$ be two Lipschitz functions, and $\tau_{1},\tau_2$ be two positive continuous functions on $[0,1]$. Let $F$ be the set 
\begin{equation*}
\left\{(x,y)\in [0,1]^2: \begin{aligned}&|T_{\beta_1}^nx-f_1(x)|<\beta_1^{-n\tau_1(x)}\\ &|T_{\beta_2}^ny-f_2(y)|<\beta_2^{-n\tau_2(y)}\end{aligned}\text{ for infinitely many }n\in\N\right\}.
\end{equation*}
 Then the Hausdorff dimension of $F$ is given by 
 \begin{equation}\label{eqdimformulas}
\dim_{\rm H}F=\begin{cases}\min\left\{\frac{2+\theta_{1}}{1+\theta_{1}}, \frac{2+\theta_{2}-\theta_{1}\log_{\beta_2}\beta_1}{1+\theta_{2}}\right\} & \quad \text{ if }  \beta_1^{1+\theta_{1}}<\beta_2,\\
\min\left\{\frac{1+\log_{\beta_2}\beta_1}{(1+\theta_{1})\log_{\beta_2}\beta_1}, \frac{2+\theta_{2}-\theta_{1}\log_{\beta_2}\beta_1}{1+\theta_{2}}\right\} &  \quad \text{ if }\beta_2\leq \beta_1^{1+\theta_{1}} \leq \beta_2^{1+\theta_{2}}, \\
\min\left\{\frac{1+\log_{\beta_2}\beta_1}{1+\theta_{2}}, \frac{(2+\theta_{1})\log_{\beta_2}\beta_1-\theta_{2}}{(1+\theta_{1})\log_{\beta_2}\beta_1}\right\} & \quad \text{ if } \beta_1^{1+\theta_{1}} > \beta_2^{1+\theta_{2}},
\end{cases}
\end{equation}
where for $i=1,2$, $\theta_i=\min_{x\in[0,1]}\tau_i(x)$. 
\end{thm}

Theorem \ref{thm2}  generalizes the result of Wang and Li \cite{WWLLL20} and is the main contribution of this paper. In view of Theorem \ref{thm2} and Theorem \ref{thm1} (and its motivation described above), it is natural to consider replacing the functions $f_1, f_2$ in Theorem \ref{thm2} by functions depending on both coordinates. For this,  we have the following partial result.

\begin{thm}\label{thm3}
Let $\beta_2\geq \beta_1>1$, $g_1,g_2:[0,1]^2\to[0,1]$ be two Lipschitz functions, and $\tau_{1},\tau_2$ be two positive continuous functions on $[0,1]$. Let $G$ be the set 
\begin{equation*}
\left\{(x,y)\in [0,1]^2: \begin{aligned}&|T_{\beta_1}^nx-g_1(x,y)|<\beta_1^{-n\tau_1(x)}\\ &|T_{\beta_2}^ny-g_2(x,y)|<\beta_2^{-n\tau_2(y)}\end{aligned}\text{ for infinitely many }n\in\N\right\}.
\end{equation*}
For $i=1,2$, set $\kappa_i=\max_{x\in[0,1]}\tau_i(x)$. If $\beta_2>\beta_1^{\kappa_1}$ and $\beta_1>\beta_2^{\kappa_2}$, then the Hausdorff dimension of $G$ is given by \eqref{eqdimformulas}.  
\end{thm}

The paper is organized as follows. In Section~\ref{S2}, we give some preliminaries which  include some elementary properties of beta transformations, and a special version of the mass transference principle from rectangles to rectangles recently proved by Wang and Wu \cite{WBWWJ21}. In Section \ref{S3}, we prove Theorem \ref{thm1}. In Section \ref{S4}, we first give the proof of Theorem \ref{thm2} and then show how to modify it to prove Theorem \ref{thm3}.

\section{Preliminaries}\label{S2}

\subsection{Properties of beta transformations}

For $\beta>1$, let $T_{\beta}$ be the $\beta$-transformation on $[0,1]$ defined by  
\[T_{\beta}x=\beta x-\lfloor \beta x\rfloor,\]
where $\lfloor \cdot \rfloor$ denotes the integral part of a real number. Then every $x\in [0,1]$ can be expressed uniquely as a finite or infinite series 
\begin{equation}\label{eqbetaexp}
x=\frac{\epsilon_1(x,\beta)}{\beta}+\frac{\epsilon_2(x,\beta)}{\beta^{2}}+\cdots+\frac{\epsilon_n(x,\beta)}{\beta^n}+\cdots,
\end{equation}
where  for $n\geq 1$, $\epsilon_n(x,\beta)=\lfloor \beta T_{\beta}^{n-1}x \rfloor$. The expression \eqref{eqbetaexp} or the sequence 
\[(\epsilon_1(x,\beta), \epsilon_2(x,\beta),\ldots)\]
is called the {\em $\beta$-expansion of $x$}, and for $n\geq 1$, $\epsilon_n(x,\beta)$ is called the $n$-th digit of $x$ (with respect to  base $\beta$).  

Note that for every $x\in [0,1]$, all its digits $\epsilon_n(x,\beta)$ are in $\{0,1,\ldots, \lceil \beta-1\rceil\}$, where $\lceil \beta-1 \rceil$ is the smallest integer not less than $\beta-1$. However, not every sequence over $\{0,1,\ldots, \lceil \beta-1\rceil\}$ is the $\beta$-expansion for some $x\in [0,1]$. We call a  finite or an infinite sequence $(\epsilon_1,\epsilon_2,\ldots)$  {\em admissible}, if there exists an $x\in [0,1]$ such  that the $\beta$-expansion of $x$ begins with $(\epsilon_1,\epsilon_2,\ldots)$. 

For $n\geq 1$, let $\Sigma_{\beta}^n$ be the set of all admissible sequences of length $n$. For the cardinality of $\Sigma_{\beta}^n$, one has the following well-known result due to R{\'e}nyi.

\begin{lem}\cite{Renyi57}\label{lemcard}
Let $\beta>1$. Then for any $n\geq 1$, 
\[\beta^n\leq \#\Sigma_{\beta}^n\leq \frac{\beta^{n+1}}{\beta-1},\]
where $\#$ denotes the cardinality of a finite set.  
\end{lem}

For any $n\geq 1$ and $w=(\epsilon_1,\ldots, \epsilon_n)\in \Sigma_{\beta}^n$, the set 
\[I_{n,\beta}(w)=\left\{x\in [0,1]: \epsilon_i(x,\beta)=\epsilon_i, 1\leq i\leq n\right\}\]
is called  {\em a cylinder of order $n$} (with respect to base $\beta$), which is a left-closed and right-open interval of length at most $\beta^{-n}$ with left endpoint 
\[\frac{\epsilon_1}{\beta}+\frac{\epsilon_2}{\beta^2}+\cdots+\frac{\epsilon_n}{\beta^n}.\]
All cylinders of order $n$ form a partition of the unit interval. That is, 
\begin{equation}\label{eqpartion01}
[0,1]=\bigcup_{w\in \Sigma_{\beta}^n}I_{n,\beta}(w)
\end{equation}
and a disjoint union. 

The following notion  plays an important role in the study of metric properties of $\beta$-expansions, which is also needed in this paper. 
\begin{de}[Full word and full cylinder]
A word $w=(\epsilon_1,\ldots, \epsilon_n)\in\Sigma_{\beta}^n$ is called a full word and $I_{n,\beta}(w)$ is called a full cylinder if 
\[|I_{n,\beta}(w)|=\frac{1}{\beta^n},\]
where $|A|$ denotes the diameter of a set $A$.  
\end{de}

The following property of full cylinders is   important in the  proofs of Theorem \ref{thm2}. 

\begin{lem}\cite[Theorem 1.2]{YBBWW14}\label{lemn1full}
For each $n\geq 1$, there exists at least one full cylinder  among every $n+1$ consecutive cylinders of order $n$.
\end{lem}

We will also need the following property of a Lipschitz function on full cylinders.

\begin{lem}\cite[Lemma 3.1]{WangW18}\label{lemxnw}
Let $f:[0,1]\to [0,1]$ be a Lipschitz function. Then for any $n\geq 1$ and full word  $w=(\epsilon_1,\ldots,\epsilon_n)\in \Sigma_{\beta}^n$,  the following hold. 

{\rm (i)} There exists a point $x^*_{n,w}$ in the closure of $I_{n,\beta}(w)$ such that $T_{\beta}^nx^*_{n,w}=f(x^*_{n,w})$ when $x^*_{n,w}\in I_{n,\beta}(w)$, and  $f(x^*_{n,w})=1$ when 
\[x^*_{n,w}=\frac{\epsilon_1}{\beta}+\frac{\epsilon_2}{\beta^2}+\cdots+\frac{\epsilon_n}{\beta^n}+\frac{1}{\beta^n}.\]

{\rm (ii)} For any $\epsilon>0$ there exists a point $x_{n,w}\in I_{n,\beta}(w)$ such that 
\[|T_{\beta}^nx_{n,w}-f(x_{n,w})|<\epsilon.\]
\end{lem}

\begin{proof}
The part {\rm (i)}  was proved in \cite[Lemma 3.1]{WangW18}. Here we only show {\rm (ii)} by using {\rm (i)}. Let $x^*_{n,w}$ be provided in {\rm (i)}. If $x^*_{n,w}\in I_{n,\beta}(w)$, then  {\rm (ii)} holds trivially by simply taking $x_{n,w}=x^*_{n,w}$. For the other case that 
\[x^*_{n,w}=\frac{\epsilon_1}{\beta}+\frac{\epsilon_2}{\beta^2}+\cdots+\frac{\epsilon_n}{\beta^n}+\frac{1}{\beta^n},\]
by the continuity of $f$ on $[0,1]$ and $T_{\beta}^n$ on $I_{n,\beta}(w)$, we can take $x_{n,w}\in I_{n,\beta}(w)$ to be a point which is sufficiently close to $x^*_{n,w}$ such that 
\[|f(x_{n,w})-f(x^*_{n,w})|=|f(x_{n,w})-1|<\epsilon/2,\]
\[|T_{\beta}^nx_{n,w}-1|=1-T_{\beta}^nx_{n,w}<\epsilon/2.\]
Then it follows that 
\[|T_{\beta}^nx_{n,w}-f(x_{n,w})|\leq|f(x_{n,w})-1|+|T_{\beta}^nx_{n,w}-1|<\epsilon.\]
\end{proof}

\subsection{Mass transference principle from rectangles to rectangles, a special version.}

Let $\R^+$ be the set of positive numbers. 
For $x=(x_1,\ldots, x_d)\in \R^d$, $\mathbf{a}=(a_1,\ldots, a_d)\in (\R^+)^d$ and $r>0$, let 
\[B(x,r^{\mathbf{a}})=\prod_{i=1}^dB(x_i, r^{a_i}).\]

In the proofs of Theorem \ref{thm2},  we will make use of a recent result of  Wang and Wu \cite{WBWWJ21}, called the mass transference principle from rectangles to rectangles in Diophantine approximation. For our purpose, we only need the following special version of Wang and Wu's result. 
\begin{lem}\label{lemmtp}\cite[Theorems 3.1-3.2]{WBWWJ21}
Let $\{J_n\}_{n\geq1}$ be a sequence of finite index sets,  $\{x_{n,\alpha}: n\geq 1,\alpha\in J_n\}$ be a sequence of points in $[0,1]^d$ and $\{r_n\}_{n\geq 1}$ be a non-increasing sequence of positive numbers tending to $0$. Let $\mathbf{a}=(a_1,\ldots, a_d), \mathbf{t}=(t_1,\ldots, t_d)\in (\mathbb{R}^+)^d$.  Set 
\[W(\mathbf{t})=\left\{x\in [0,1]^d: x\in B(x_{n,\alpha},r_n^{\mathbf{a}+\mathbf{t}})\text{ for infinitely many }n\in\N\text{ and }\alpha\in J_n\right\}.\] If for all large $n$, the set
\begin{equation}\label{eqlocalubiquity}
\left\{x\in [0,1]^d: x\in \bigcup_{\alpha\in J_n}B(x_{n,\alpha}, r_n^{\mathbf{a}})\right\}
\end{equation} 
is of full Lebesgue measure, then for any ball $B\subset [0,1]^d$, 
\[\dim_{\rm H}B\cap W(\mathbf{t})\geq \min_{A\in \mathcal{A}}\left\{\#\mathcal{K}_1+\#\mathcal{K}_2+\frac{\sum_{k\in\mathcal{K}_3}a_k-\sum_{k\in\mathcal{K}_2}t_k}{A}\right\},\]
where 
\[\mathcal{A}=\{a_i,a_i+t_i: 1\leq i\leq d\},\]
and for each $A\in \mathcal{A}$, the sets $\mathcal{K}_1,\mathcal{K}_2,\mathcal{K}_3$ form a partition of $\{1,\ldots,d\}$ defined as 
\[\mathcal{K}_1=\{k:a_k\geq A\},\ \mathcal{K}_2=\{k: a_k+t_k\leq A\}\setminus \mathcal{K}_1,\ \mathcal{K}_3=\{1,\ldots,d\}\setminus(\mathcal{K}_1\cup\mathcal{K}_2).\]
\end{lem}

We remark that the condition \eqref{eqlocalubiquity}  above  is to guarantee the  local ubiquity for rectangles (\cite[Definition 3.2]{WBWWJ21}) needed in \cite[Theorems 3.1-3.2]{WBWWJ21}.

\section{Proof of Theorem \ref{thm1}}\label{S3}

We first prove  the Lebesgue measure part of the theorem.  For $y\in [0,1]$, let 
\begin{equation*}\label{eqDy}
D_y=\left\{x\in [0,1]: |T_{\beta}^nx-f(x,y)|<\varphi(n)\text{ for infinitely many }n\in \N\right\}.
\end{equation*}
Since $f:[0,1]^2\to [0,1]$ is a Lipschitz function, so is the function $x\mapsto f(x,y)$ for every $y\in [0,1]$. Hence by \cite[Theorem 1.6]{FLBWJW21},  for every $y\in [0,1]$,  
\[\mathcal{L}(D_y)=\begin{cases}0 & \text{ if }\sum_{n=1}^{\infty}\varphi(n)<\infty,\\
1& \text{ if }\sum_{n=1}^{\infty}\varphi(n)=\infty.
\end{cases}\]
Then by the Fubini's theorem,
\[\mathcal{L}^2(W_{\beta}(f,\varphi))=\int_{0}^1\int_0^1\mathbbm{1}_{W_{\beta}(f,\varphi)}((x,y))dxdy=\int_0^1\int_0^1\mathbbm{1}_{D_y}(x)dxdy=\int_0^1\mathcal{L}(D_y)dy,\]
where $\mathbbm{1}_A$ denotes the characteristic function of a set $A$. Therefore, we have 
\[\mathcal{L}^2(W_{\beta}(f,\varphi))=\begin{cases}0 & \text{ if } \sum_{n=1}^{\infty}\varphi(n)<\infty,\\
1& \text{ if } \sum_{n=1}^{\infty}\varphi(n)=\infty.
\end{cases}\]
This proves the Lebesgue measure part of the theorem. 

Next we prove the Hausdorff dimension part. Suppose that $\sum_{n=1}^{\infty}\varphi(n)<\infty$. Then without loss of generality we can assume that $\varphi(n)<1$ for all $n\in\N$. Again since for every $y\in [0,1]$,  the function $x\mapsto f(x,y)$ is Lipschitz,  it follows from \cite[Theorem 1.1]{WangW18} that for any $y\in [0,1]$, 
\[\dim_{\rm H}D_{y}=\frac{1}{1+\alpha}\quad \text{ with }\alpha=\liminf_{n\to\infty}\frac{\log_{\beta}\varphi(n)^{-1}}{n}.\]
Then by  \cite[Corollary 7.12]{falconer1}, 
\[\dim_{\rm H}W_{\beta}(f,\varphi)\geq 1+\frac{1}{1+\alpha}.\]
In the following we prove the converse inequality.

For $n\in\N$, let 
\[W_n=\left\{(x,y)\in [0,1]^2: |T_{\beta}^nx-f(x,y)|<\varphi(n)\right\}.\]
Then 
\begin{equation}\label{eqWbetaflims}
W_{\beta}(f,\varphi)=\bigcap_{N=1}^{\infty}\bigcup_{n=N}^{\infty}W_n.
\end{equation}
Let 
\[J_{n,\beta}(k)=\left[\frac{k\varphi(n)}{\beta^n}, \frac{(k+1)\varphi(n)}{\beta^n}\right]\cap [0,1], \quad \text{ for } k=0,1,\ldots, \left\lfloor\frac{\beta^n}{\varphi(n)}\right\rfloor.\]
Then clearly 
\[[0,1]=\bigcup_{0\leq k\leq \left\lfloor\frac{\beta^n}{\varphi(n)}\right\rfloor}J_{n,\beta}(k).\]
Hence we have 
\[[0,1]^2=\bigcup_{w\in \Sigma_{\beta}^n}\bigcup_{0\leq k\leq \left\lfloor\frac{\beta^n}{\varphi(n)}\right\rfloor}I_{n,\beta}(w)\times J_{n,\beta}(k).\]
Therefore,
\begin{equation}\label{eqWnthm1}
W_n=\bigcup_{w\in \Sigma_{\beta}^n}\bigcup_{0\leq k\leq \left\lfloor\frac{\beta^n}{\varphi(n)}\right\rfloor}\left\{(x,y)\in I_{n,\beta}(w)\times J_{n,\beta}(k): |T_{\beta}^nx-f(x,y)|<\varphi(n)\right\}.
\end{equation}

Since $f:[0,1]^2\to [0,1]$ is Lipschitz, there exists $L>0$ such that for any $x,y, x', y'\in [0,1]$,
\begin{equation}\label{eqflip}
|f(x,y)-f(x', y')|\leq L\|(x-x',y-y')\|, 
\end{equation}
where  $\|\cdot \|$ is the Euclidean norm. 
Let  $w\in \Sigma_{\beta}^n$, $0\leq k\leq \left\lfloor\frac{\beta^n}{\varphi(n)}\right\rfloor$, and  $(x,y)\in I_{n,\beta}(\omega)\times J_{n,\beta}(k)$. If $(x,y)\in W_n$, then by \eqref{eqflip} we see that for large $n$,
\begin{align} 
|T_{\beta}^nx-f(x,k\varphi(n)/\beta^n)|&\leq |T_{\beta}^nx-f(x,y)|+|f(x,y)-f(x,k\varphi(n)/\beta^n)|\nonumber\\
&\leq \varphi(n)+\frac{L\varphi(n)}{\beta^n}\nonumber\\
&<2\varphi(n). \label{equpbd2vn}
\end{align}
Therefore, for large $n$, we have 
\begin{align*}
&\left\{(x,y)\in I_{n,\beta}(w)\times J_{n,\beta}(k): |T_{\beta}^nx-f(x,y)|<\varphi(n)\right\}\\
&\quad \subset \left\{x\in I_{n,\beta}(w): |T_{\beta}^nx-f(x,k\varphi(n)/\beta^n)|<2\varphi(n)\right\}\times J_{n,\beta}(k)\\
&\quad :=\tilde{I}_{n,\beta}(w,k)\times J_{n,\beta}(k),
\end{align*}
and thus by \eqref{eqWnthm1},
\begin{equation}\label{eqcoverWn}
W_n\subset \bigcup_{w\in \Sigma_{\beta}^n}\bigcup_{0\leq k\leq \left\lfloor\frac{\beta^n}{\varphi(n)}\right\rfloor}\tilde{I}_{n,\beta}(w,k)\times J_{n,\beta}(k).
\end{equation}

Below we estimate the diameter of $\tilde{I}_{n,\beta}(w,k)\times J_{n,\beta}(k)$. To this end, let $x_1,x_2\in \tilde{I}_{n,\beta}(w,k)$. Then by \eqref{equpbd2vn} and \eqref{eqflip}, for large $n$,  
\begin{align*}
4\varphi(n)&>|T_{\beta}^nx_1-f(x_1, k\varphi(n)/\beta^n)|+|T_{\beta}^nx_2-f(x_2, k\varphi(n)/\beta^n)|\\
&\geq |T_{\beta}^nx_1-T_{\beta}^nx_2|-|f(x_1, k\varphi(n)/\beta^n)-f(x_2, k\varphi(n)/\beta^n)|\\
& \geq (\beta^n-L)|x_1-x_2|\\
&\geq \frac{\beta^n}{2}|x_1-x_2|.
\end{align*}
This implies that for all large $n$, the diameter of $\tilde{I}_{n,\beta}(w,k)$ satisfies that 
 $$|\tilde{I}_{n,\beta}(w,k)|\leq \frac{8\varphi(n)}{\beta^n}.$$
As a consequence, for all large $n$, 
\[|\tilde{I}_{n,\beta}(w,k)\times J_{n,\beta}(k)|\leq \frac{9\varphi(n)}{\beta^n}. \]

By \eqref{eqWbetaflims} and \eqref{eqcoverWn}, for every $N\in\N$,  the family 
\[\left\{\tilde{I}_{n,\beta}(w,k)\times J_{n,\beta}(k): n\geq N, w\in \Sigma_{\beta}^n, 0\leq k\leq \left\lfloor\frac{\beta^n}{\varphi(n)}\right\rfloor\right\}\]
is a covering  of $W_{\beta}(f,\varphi)$. Recall that $\alpha=\liminf_{n\to\infty}\frac{\log_{\beta}\varphi(n)^{-1}}{n}$. Let $s>1+\frac{1}{1+\alpha}$. Then we have 
\begin{align*}
\mathcal{H}^s_{\infty}(W_{\beta}(f,\varphi))&\leq \liminf_{N\to\infty} \sum_{n\geq N}\sum_{w\in \Sigma_{\beta}^n}\sum_{0\leq k\leq \left\lfloor\frac{\beta^n}{\varphi(n)}\right\rfloor}|\tilde{I}_{n,\beta}(w,k)\times J_{n,\beta}(k)|^s\\
&\leq \liminf_{N\to\infty}\sum_{n\geq N}\sum_{w\in \Sigma_{\beta}^n}\sum_{0\leq k\leq \left\lfloor\frac{\beta^n}{\varphi(n)}\right\rfloor}\left(\frac{9\varphi(n)}{\beta^n}\right)^s\\
&\leq \liminf_{N\to\infty}\sum_{n\geq N}\frac{\beta^{n+1}}{\beta-1}\cdot\frac{2\beta^n}{\varphi(n)}\cdot \left(\frac{9\varphi(n)}{\beta^n}\right)^s \quad \text{(by Lemma \ref{lemcard})}\\
&=0,
\end{align*}
here and afterwards $\mathcal{H}^s_{\infty}$ denotes the $s$-dimensional  Hausdorff content \cite{BishopPeres17}. 
Since $s>1+\frac{1}{1+\alpha}$ is arbitrary, 
it follows that 
\[\dim_{\rm H}W_{\beta}(f,\varphi)\leq 1+\frac{1}{1+\alpha}.\]
This proves the Hausdorff dimension part and   completes the proof of Theorem \ref{thm1}.

\section{Proofs of Theorems \ref{thm2}-\ref{thm3}}\label{S4}

\subsection{Proof of Theorem \ref{thm2}}\label{subs1}

{\bf The upper bounds part}. Recall the formula \eqref{eqpartion01}. So for every $n\in \N$, 
\[[0,1]^2=\bigcup_{w\in \Sigma_{\beta_1}^n, v\in \Sigma_{\beta_2}^n}I_{n,\beta_1}(w)\times I_{n,\beta_2}(v).\]
For $w\in \Sigma_{\beta_1}^n$ and  $v\in \Sigma_{\beta_2}^n$, let 
\begin{equation}\label{eqJn1b1w}
J_{n,\beta_1}(w)=\left\{x\in I_{n,\beta_1}(w): |T_{\beta_1}^nx-f_1(x)|<\beta_1^{-n\tau_1(x)}\right\},
\end{equation}
\[J_{n,\beta_2}(v)=\left\{y\in I_{n,\beta_2}(v): |T_{\beta_2}^ny-f_2(y)|<\beta_2^{-n\tau_2(y)}\right\}.\]
Then we have 
\begin{align}
F&=\bigcap_{N=1}^{\infty}\bigcup_{n=N}^{\infty}\left\{(x,y)\in [0,1]^2:\begin{aligned}|T_{\beta_1}^nx-f_1(x)|<\beta_1^{-n\tau_1(x)}\\|T_{\beta_2}^ny-f_2(y)|<\beta_2^{-n\tau_2(y)}\end{aligned}\right\}\nonumber\\
&=\bigcap_{N=1}^{\infty}\bigcup_{n=N}^{\infty}\bigcup_{w\in \Sigma_{\beta_1}^n, v\in \Sigma_{\beta_2}^n}J_{n,\beta_1}(w)\times J_{n,\beta_2}(v).\label{eqWb1b2}
\end{align}

Fix $w\in \Sigma_{\beta_1}^n$ and  $v\in \Sigma_{\beta_2}^n$. In the following we estimate the diameters of $J_{n,\beta_1}(w)$ and $J_{n,\beta_2}(v)$. For $i=1,2$, let $\theta_{i}=\min_{x\in[0,1]}\tau_i(x).$ 
Since $f_1:[0,1]\to [0,1]$ is Lipschitz, there exists $L>0$ such that 
\begin{equation}\label{eqLif1x}
|f_1(x)-f_1(x')|\leq L|x-x'|\quad \text{ for any }x,x'\in [0,1].
\end{equation}
Let $x, x'\in J_{n,\beta_1}(w)$. Then by \eqref{eqJn1b1w} and \eqref{eqLif1x}, 
\begin{align}
2\beta_1^{-n\theta_{1}}&\geq |T_{\beta_1}^nx-f_1(x)|+|T_{\beta_1}^nx'-f_1(x')|\nonumber\\
&\geq |T_{\beta_1}^nx-T_{\beta_1}^nx'|-|f_1(x)-f_1(x')|\nonumber\\
&\geq (\beta_1^n-L)|x-x'|. \label{eqbeta1nLxx}
\end{align}
Since $\beta_1^n-L>\frac{\beta_1^n}{2}$ for large $n$, \eqref{eqbeta1nLxx} implies that  for large $n$,
\begin{equation}\label{Jnbeta1}
|J_{n,\beta_1}(w)|\leq 4\beta_1^{-n(1+\theta_{1})}.
\end{equation} 
Similarly, for large $n$ we have 
\begin{equation}\label{Jnbeta2}
|J_{n,\beta_2}(v)|\leq 4\beta_2^{-n(1+\theta_{2})}.
\end{equation}

From \eqref{eqWb1b2} we see that for every $N\in\N$, the family 
\begin{equation}\label{eqcovfam}
\left\{J_{n,\beta_1}(w)\times J_{n,\beta_2}(v): n\geq N, w\in \Sigma_{\beta_1}^n, v\in \Sigma_{\beta_2}^n\right\}
\end{equation}
is a covering of $F$. In the following, we obtain upper bounds for the Hausdorff dimension of $F$ by considering three cases separately.

{\em Case 1. $\beta_1^{1+\theta_1}<\beta_2$}. Note that in this case, $$4\beta_2^{-n(1+\theta_{2})}\leq 4\beta_2^{-n}\leq 4\beta_1^{-n(1+\theta_{1})}, \quad \forall n\in\N.$$ 
Hence  each   $J_{n,\beta_1}(w)\times J_{n,\beta_2}(v)$ in the family \eqref{eqcovfam} can be covered by at most 
\[2\times \frac{4\beta_1^{-n(1+\theta_{1})}}{4\beta_2^{-n(1+\theta_{2})}}= 2\beta_2^{n\left((1+\theta_{2})-(1+\theta_{1})\log_{\beta_2}\beta_1\right)}\]
many squares of side length $4\beta_2^{-n(1+\theta_{2})}$. Thus for $s>\frac{2+\theta_{2}-\theta_{1}\log_{\beta_2}\beta_1}{1+\theta_{2}}$, we have 
\begin{align*}
\mathcal{H}^{s}_{\infty}(F)&\leq \liminf_{N\to\infty}\sum_{n=N}^{\infty}\sum_{w\in \Sigma_{\beta_1}^{n}, v\in \Sigma_{\beta_2}^{n}}2\beta_2^{n\left((1+\theta_{2})-(1+\theta_{1})\log_{\beta_2}\beta_1\right)}\left(4\sqrt{2}\beta_2^{-n(1+\theta_{2})}\right)^s\\
& \ll \liminf_{N\to\infty}\sum_{n=N}^{\infty}\beta_1^n\beta_2^n\beta_2^{n\left((1+\theta_{2})(1-s)-(1+\theta_{1})\log_{\beta_2}\beta_1\right)} \quad \ \text{(by Lemma \ref{lemcard})}\\
&= \liminf_{N\to\infty}\sum_{n=N}^{\infty}\beta_2^{n\left(2+\theta_{2}-\theta_{1}\log_{\beta_2}\beta_1-(1+\theta_{2})s\right)}\\
& =0.
\end{align*}
Hence 
\begin{equation}\label{equpp1}
\dim_{\rm H}F\leq \frac{2+\theta_{2}-\theta_{1}\log_{\beta_2}\beta_1}{1+\theta_{2}}.
\end{equation}

On the other hand, since $4\beta_2^{-n(1+\theta_{2})}\leq 4\beta_1^{-n(1+\theta_{1})}$,  each  $J_{n,\beta_1}(w)\times J_{n,\beta_2}(v)$ can be covered by a single square of side length $4\beta_1^{-n(1+\theta_{1})}$. Moreover,  since $\beta_2^{-n}<4\beta_1^{-n(1+\theta_{1})}$, we see that each such square covers at least 
\[\frac{1}{2}\times \frac{4\beta_1^{-n(1+\theta_{1})}}{\beta_2^{-n}}=8\beta_2^{n(1-(1+\theta_1)\log_{\beta_2}\beta_1)}:=k_n(w)\]
many sets $J_{n,\beta_1}(w)\times J_{n,\beta_2}(v_i)$ ($i=1,\ldots,k_n(w)$) such that $I_{n,\beta_2}(v_1),\ldots, I_{n,\beta_2}(v_{k_n(w)})$ are consecutive cylinders of order $n$. Therefore, the set $$\bigcup_{w\in \Sigma_{\beta_1}^n, v\in \Sigma_{\beta_2}^n}J_{n,\beta_1}(w)\times J_{n,\beta_2}(v)$$ can be covered by 
\[\frac{\beta_1^n}{\beta_1-1}\cdot\frac{\beta_2^n}{\beta_2-1}\cdot\frac{1}{8\beta_2^{n(1-(1+\theta_1)\log_{\beta_2}\beta_1)}}\]
many squares of side length $4\beta_1^{-n(1+\theta_{1})}$. As a consequence, for $s>\frac{2+\theta_1}{1+\theta_1}$, we have 
\begin{align*}
\mathcal{H}^s_{\infty}(F)&\ll \liminf_{N\to\infty}\sum_{n=N}^{\infty}\beta_1^n\beta_2^n\beta_2^{-n(1-(1+\theta_1)\log_{\beta_2}\beta_1)}\beta_1^{-n(1+\theta_{1})s}\\
&=\liminf_{N\to\infty}\sum_{n=N}^{\infty}\beta_2^{n\left(2+\theta_1-(1+\theta_{1})s\right)\log_{\beta_2}\beta_1}\\
&=0. 
\end{align*}
Hence 
\begin{equation*}\label{equpp49}
\dim_{\rm H}F\leq \frac{2+\theta_1}{1+\theta_1}.
\end{equation*}
This combined with \eqref{equpp1} gives the desired upper bound for $\dim_{\rm H}F$ when $\beta_1^{1+\theta_1}<\beta_2$. 

{\em Cases 2. $\beta_2\leq \beta_1^{1+\theta_1}\leq \beta_2^{1+\theta_2}$}.  Since $4\beta_2^{-n(1+\theta_{2})}\leq 4\beta_1^{-n(1+\theta_{1})}$,  as in Case 1 we see that \eqref{equpp1} still holds. Again, each $J_{n,\beta_1}(w)\times J_{n,\beta_2}(v)$ in \eqref{eqcovfam}  can be covered by a single square of side length $4\beta_1^{-n(1+\theta_{1})}$. Thus for $s>\frac{1+\log_{\beta_2}\beta_1}{(1+\theta_{1})\log_{\beta_2}\beta_1}$, we have 
\begin{align*}
\mathcal{H}^s_{\infty}(F)&\leq \liminf_{N\to\infty}\sum_{n=N}^{\infty}\sum_{w\in \Sigma_{\beta_1}^{n},v\in \Sigma_{\beta_2}^{n}}\left(4\sqrt{2}\beta_1^{-n(1+\theta_{1})}\right)^s\\
&\ll \liminf_{N\to\infty}\sum_{n=N}^{\infty}\beta_1^n\beta_2^n\beta_1^{-ns(1+\theta_{1})}\\
&=\liminf_{N\to\infty}\sum_{n=N}^{\infty}\beta_2^{n\left((1+\log_{\beta_2}\beta_1)-s(1+\theta_{1})\log_{\beta_2}\beta_1\right)}\\
&=0,
\end{align*}
which implies that 
\begin{equation*}\label{equpp2}
\dim_{\rm H}F\leq \frac{1+\log_{\beta_2}\beta_1}{(1+\theta_{1})\log_{\beta_2}\beta_1}.
\end{equation*}
Combining this with \eqref{equpp1} yields the desired upper bound for $\dim_{\rm H}F$ in Case 2.  

{\em Case 3. $\beta_1^{1+\theta_1}>\beta_2^{1+\theta_2}$}. In this case, we have 
$$4\beta_2^{-n(1+\theta_{2})}>4\beta_1^{-n(1+\theta_{1})}, \quad \forall n\in\N.$$ 
Then by a similar argument  as in Case 2, it is easily seen that 
\[\dim_{\rm H}F\leq \min\left\{\frac{1+\log_{\beta_2}\beta_1}{1+\theta_{2}}, \frac{(2+\theta_{1})\log_{\beta_2}\beta_1-\theta_{2}}{(1+\theta_{1})\log_{\beta_2}\beta_1}\right\}.\]

{\bf The lower bounds part}.  Fix a full word $w\in \Sigma_{\beta_1}^n$.  Then by Lemma \ref{lemxnw}(ii),  there exists a point $x_{n,w}\in I_{n,\beta_1}(w)$  such that 
\[|T_{\beta_1}^nx_{n,w}-f_1(x_{n,w})|<\frac{1}{2}\beta_1^{-n\kappa_1},\]
where $\kappa_i:=\max_{x\in [0,1]}\tau_i(x)$ for $i=1,2$. 
Thus for any $x\in I_{n,\beta_1}(w)$ and all large $n$, \
\begin{align*}
|T_{\beta_1}^nx-f_1(x)|-\frac{1}{2}\beta_1^{-n\kappa_1}&\leq  |T_{\beta_1}^nx-f_1(x)|-|T_{\beta_1}^nx_{n,w}-f_1(x_{n,w})|\\
&\leq |T_{\beta_1}^nx-T_{\beta_1}^nx_{n,w}|+|f_1(x)-f_1(x_{n,w})|\\
&\leq (\beta_1^n+L)|x-x_{n,w}| \quad \ \qquad (\text{by \eqref{eqLif1x}})\\
&\leq 2\beta_1^n|x-x_{n,w}|. 
\end{align*}
Hence if $|x-x_{n,w}|< \frac{1}{4}\beta_{1}^{-n(1+\tau_1(x))}$, then
\[|T_{\beta_1}^nx-f_1(x)|<2\beta_1^n\times \frac{1}{4}\beta_{1}^{-n(1+\tau_1(x))}+\frac{1}{2}\beta_1^{-n\kappa_1}\leq \beta_{1}^{-n\tau_1(x)}.\] 
This implies that 
\[J_{n,\beta_1}(w)\supset \left\{x\in I_{n,\beta_1}(w): |x-x_{n,w}|< \frac{1}{4}\beta_{1}^{-n(1+\tau_1(x))}\right\}:=\tilde{J}_{n,\beta_1}(w).\]
Similarly, for any full word $v\in \Sigma_{\beta_2}^{n}$ and for all large $n$, there exists a point $y_{n,v}\in I_{n,\beta_2}(v)$ such  that 
\[J_{n,\beta_2}(v)\supset \left\{y\in I_{n,\beta_2}(v): |y-y_{n,v}|<\frac{1}{4}\beta_{2}^{-n(1+\tau_{2}(y))}\right\}:=\tilde{J}_{n,\beta_2}(v).\]
Therefore, by \eqref{eqWb1b2} we have 
\begin{equation}\label{eqWsubset}
F\supset \bigcap_{N=1}^{\infty}\bigcup_{n=N}^{\infty}\bigcup_{w\in \Sigma_{\beta_1}^n, v\in \Sigma_{\beta_2}^n\text{ full}}\tilde{J}_{n,\beta_1}(w)\times \tilde{J}_{n,\beta_2}(v):=\widetilde{F}.
\end{equation}

Fix $\epsilon>0$. 
Since $\tau_1,\tau_2$ are  continuous functions on $[0,1]$, there exists a ball $B\subset [0,1]^2$ such that for any $(x,y)\in B$, 
\[\tau_1(x)\leq \theta_{1}+\epsilon/2, \quad \tau_2(y)\leq \theta_{2}+\epsilon/2.\]
Let $n\in\N$ be large so that 
\begin{equation}\label{eqlargen}
\beta_1^{-n\epsilon/2}\leq \frac{1}{8} \quad \text{ and } \quad  n+1\leq \beta_1^{n\epsilon}.
\end{equation}
Then for any full words $w\in \Sigma_{\beta_1}^n$ and $v\in \Sigma_{\beta_2}^n$,  we have  
\begin{equation}\label{eqBcaptJ}
B\cap \left(\tilde{J}_{n,\beta_1}(w)\times \tilde{J}_{n,\beta_2}(v)\right) \supset B\cap \left(J^*_{n,\beta_1}(w)\times J^*_{n,\beta_2}(v)\right),
\end{equation}
where 
\[J^*_{n,\beta_1}(w)=\left\{x\in I_{n,\beta_1}(w): |x-x_{n,w}|< \frac{1}{4}\beta_{1}^{-n(1+\theta_{1}+\epsilon/2)}\right\},\]
\[J^*_{n,\beta_2}(v)=\left\{y\in I_{n,\beta_2}(v): |y-y_{n,v}|<\frac{1}{4}\beta_{2}^{-n(1+\theta_{2}+\epsilon/2)}\right\}.\]
Note that $J^*_{n,\beta_1}(w)$ contains a ball centered in $I_{n,\beta_1}(w)$ of radius $\frac{1}{8}\beta_{1}^{-n(1+\theta_{1}+\epsilon/2)}$, and $J^*_{n,\beta_2}(v)$ contains a ball centered in $I_{n,\beta_2}(v)$ of radius $\frac{1}{8}\beta_{2}^{-n(1+\theta_{2}+\epsilon/2)}$. Hence we see from \eqref{eqlargen} and \eqref{eqBcaptJ} that 
\begin{align*}
&B\cap \left(\tilde{J}_{n,\beta_1}(w)\times \tilde{J}_{n,\beta_2}(v)\right)\supset B\cap \left(B\left(\overline{x}_{n,w}, \beta_1^{-n(1+\theta_{1}+\epsilon)}\right)\times B\left(\overline{y}_{n,v}, \beta_2^{-n(1+\theta_{2}+\epsilon)}\right)\right)
\end{align*}
for some $\overline{x}_{n,w}\in I_{n,\beta_1}(w)$ and $\overline{y}_{n,v}\in I_{n,\beta_2}(v)$. 
Therefore, 
\begin{equation}\label{supWtilde}
B\cap \widetilde{F}\supset B\cap W(\mathbf{t}),
\end{equation}
where 
\begin{equation}\label{eqwt}
W(\mathbf{t})=\bigcap_{N=1}^{\infty}\bigcup_{n=N}^{\infty}\bigcup_{w\in \Sigma_{\beta_1}^n, v\in \Sigma_{\beta_2}^n\text{ full}}B\left(\overline{x}_{n,w}, \beta_2^{-n(1+\theta_{1}+\epsilon)\log_{\beta_2}\beta_1}\right)\times B\left(\overline{y}_{n,v}, \beta_2^{-n(1+\theta_{2}+\epsilon)}\right),
\end{equation}
and $$\mathbf{t}=(t_1, t_2)=((\theta_{1}+2\epsilon)\log_{\beta_2}\beta_1, \theta_{2}+2\epsilon).$$

From Lemma \ref{lemn1full} we know that for every $x\in [0,1]$, among any $n+1$ consecutive cylinders of order $n$ around $x$, there is at least one full cylinder. So, there exists a full word $w\in \Sigma_{\beta_1}^n$  such that 
\[|x-\overline{x}_{n,w}|\leq (n+1)\beta_1^{-n}\leq\beta_1^{-n(1-\epsilon)}=\beta_2^{-n(1-\epsilon)\log_{\beta_2}\beta_1}.\]
Thus
\[[0,1]\subset \bigcup_{w\in\Sigma_{\beta_1}^n\text{ full}}B\left(\overline{x}_{n,w}, \beta_2^{-n(1-\epsilon)\log_{\beta_2}\beta_1}\right).\]
Similarly, 
\[[0,1]\subset \bigcup_{v\in\Sigma_{\beta_2}^n\text{ full}}B\left(\overline{y}_{n,v}, \beta_2^{-n(1-\epsilon)}\right).\]
Therefore, for all large $n$, the set 
\begin{equation}\label{eqWa01}
\left\{z\in [0,1]^2: z\in \bigcup_{w\in \Sigma_{\beta_1}^n,v\in \Sigma_{\beta_2}^n\text{ full}}B\left(\overline{x}_{n,w}, \beta_2^{-n(1-\epsilon)\log_{\beta_2}\beta_1}\right)\times B\left(\overline{y}_{n,v}, \beta_2^{-n(1-\epsilon)}\right)\right\}
\end{equation}
is of full Lebesgue measure (indeed it equals $[0,1]^2$). Let   $$\mathbf{a}=(a_1,a_2)=\left((1-\epsilon)\log_{\beta_2}\beta_1, 1-\epsilon\right).$$
Then
\[\mathbf{a}+\mathbf{t}=(a_1+t_1,a_2+t_2)=((1+\theta_{1}+\epsilon)\log_{\beta_2}\beta_1, 1+\theta_{2}+\epsilon).\]

Now by \eqref{eqwt}, \eqref{eqWa01} and Lemma \ref{lemmtp}, we have 
\begin{equation}\label{eqdimbWt}
\dim_{\rm H}B\cap W(\mathbf{t})\geq \min_{A\in\mathcal{A}}\left\{\#\mathcal{K}_1+\#\mathcal{K}_2+\frac{\sum_{k\in \mathcal{K}_3}a_k-\sum_{k\in \mathcal{K}_2}t_k}{A}\right\}:=s,
\end{equation}
where 
\[\mathcal{A}=\{a_k,a_k+t_k: k=1,2\}=\{(1-\epsilon)\log_{\beta_2}\beta_1, (1+\theta_{1}+\epsilon)\log_{\beta_2}\beta_1, 1-\epsilon, 1+\theta_{2}+\epsilon\},\]
and for each $A\in\mathcal{A}$ the sets $\mathcal{K}_1, \mathcal{K}_2, \mathcal{K}_3$ form a partition of $\{1,2\}$ defined as 
\[\mathcal{K}_1=\{k: a_k\geq A\}, \quad  \mathcal{K}_2=\{k: a_k+t_k\leq A\}\setminus \mathcal{K}_1, \quad \mathcal{K}_3=\{1,2\}\setminus(\mathcal{K}_1\cup \mathcal{K}_2). \]
For convenience,  write 
\[s_A=\#\mathcal{K}_1+\#\mathcal{K}_2+\frac{\sum_{k\in \mathcal{K}_3}a_k-\sum_{k\in \mathcal{K}_2}t_k}{A}\qquad \text{ for }A\in\mathcal{A}.\]
In the following, we evaluate $s$ in \eqref{eqdimbWt}. We consider three cases separately.

When  $\beta_1^{1+\theta_1}<\beta_2$, we let $\epsilon>0$ be small enough so that   $(1+\theta_{1}+\epsilon)\log_{\beta_2}\beta_1\leq 1-\epsilon$. Then by definition and  a simple calculation, we see that  
\begin{equation*}
s_A=\begin{cases}2& \quad \text{ if }A= (1-\epsilon)\log_{\beta_2}\beta_1,\\
\frac{2+\theta_{1}}{1+\theta_{1}+\epsilon}& \quad \text{ if }A=(1+\theta_{1}+\epsilon)\log_{\beta_2}\beta_1,\\
2-\frac{(\theta_{1}+2\epsilon)\log_{\beta_2}\beta_1}{1-\epsilon}& \quad  \text{ if }A=1-\epsilon,\\
\frac{2+\theta_{2}-(\theta_{1}+2\epsilon)\log_{\beta_2}\beta_1}{1+\theta_{2}+\epsilon} &\quad \text{ if }A=1+\theta_{2}+\epsilon.
\end{cases}
\end{equation*}
Since $(1+\theta_{1}+\epsilon)\log_{\beta_2}\beta_1\leq 1-\epsilon$, we have   
\begin{equation*}\label{eq2frac}
2-\frac{(\theta_{1}+2\epsilon)\log_{\beta_2}\beta_1}{1-\epsilon}\geq 2-\frac{\theta_{1}+2\epsilon}{1+\theta_{1}+\epsilon}=\frac{2+\theta_{1}}{1+\theta_{1}+\epsilon}.
\end{equation*}
Hence in this case, 
\begin{align*}
s=\min_{A\in\mathcal{A}}s_A=\min\left\{\frac{2+\theta_{1}}{1+\theta_{1}+\epsilon}, \frac{2+\theta_{2}-(\theta_{1}+2\epsilon)\log_{\beta_2}\beta_1}{1+\theta_{2}+\epsilon}\right\}.
\end{align*}
By this, \eqref{eqWsubset}, \eqref{supWtilde} and \eqref{eqdimbWt} (and let $\epsilon\to0$), we obtain the desired lower bound for $\dim_{\rm H}F$ when $\beta_1^{1+\theta_1}<\beta_2$. 

For the other two cases that $\beta_2\leq \beta_1^{1+\theta_1}\leq \beta_2^{1+\theta_2}$ and  $\beta_1^{1+\theta_1}>\beta_2^{1+\theta_2}$, we let $\epsilon>0$ be  small enough so that 
\[\begin{cases}
1-\epsilon<(1+\theta_{1}+\epsilon)\log_{\beta_2}\beta_1<1+\theta_{2}+\epsilon&\quad \text{ if }\beta_2\leq \beta_1^{1+\theta_1}\leq \beta_2^{1+\theta_2},\\
 (1+\theta_{1}+\epsilon)\log_{\beta_2}\beta_1\geq 1+\theta_{2}+\epsilon &\quad  \text{ if } \beta_1^{1+\theta_1}>\beta_2^{1+\theta_2}.
\end{cases}\]
Then in each of these two cases,  $s_A$ for $A\in\mathcal{A}$ (and thus $s$) can be easily evaluated. 	Indeed, we have 
\begin{equation*}
s=\begin{cases}
\min\left\{\frac{(1-\epsilon)(1+\log_{\beta_2}\beta_1)}{(1+\theta_{1}+\epsilon)\log_{\beta_2}\beta_1}, \frac{2+\theta_{2}-(\theta_{1}+2\epsilon)\log_{\beta_2}\beta_1}{1+\theta_{2}+\epsilon}\right\}& \quad \text{ if }\beta_2\leq \beta_1^{1+\theta_1}\leq \beta_2^{1+\theta_2},\\
\min\left\{\frac{(1-\epsilon)(1+\log_{\beta_2}\beta_1)}{1+\theta_{2}+\epsilon}, \frac{(2+\theta_{1})\log_{\beta_2}\beta_1-(\theta_{2}+2\epsilon)}{(1+\theta_{1}+\epsilon)\log_{\beta_2}\beta_1}\right\} & \quad \text{ if } \beta_1^{1+\theta_1}>\beta_2^{1+\theta_2}.
\end{cases}
\end{equation*}
Then a similar argument as above yields the desired lower bounds for $\dim_{\rm H}F$ in these two cases. This completes the proof of Theorem \ref{thm2}.

\subsection{Proof of Theorem \ref{thm3}}
Since the proof of Theorem \ref{thm3} is similar to that of Theorem \ref{thm2}, in this subsection we only point out the modifications of Subsection \ref{subs1} needed to prove Theorem \ref{thm3}.

{\bf The  upper bounds part}. We prove this part under the weaker assumption that $\beta_2\geq \beta_1^{\theta_1}$ and $\beta_1\geq \beta_2^{\theta_2}$. Since $g_1, g_2:[0,1]^2\to [0,1]$ are Lipschitz functions, there exists $L>0$ such that for any $(x,y), (x',y')\in [0,1]^2$, 
\begin{equation}\label{eqgiLip}
|g_i(x,y)-g_i(x',y')|\leq L\|(x-x', y-y')\|, \quad i=1,2.
\end{equation}

For $w\in \Sigma_{\beta_1}^{n}$ and $v\in \Sigma_{\beta_2}^n$, let $a^1_{n,w}$, $a^2_{n,v}$ be the left endpoints of  $I_{n,\beta_1}(w)$ and $I_{n,\beta_2}(v)$, respectively. Let $(x,y)\in I_{n,\beta_1}(w)\times I_{n,\beta_2}(v)$ such that 
 \[|T_{\beta_1}^nx-g_1(x,y)|<\beta_1^{-n\tau_1(x)}, \quad |T_{\beta_2}^ny-g_2(x,y)|<\beta_2^{-n\tau_2(y)}.\]
 Then we have 
 \begin{align*}
 |T_{\beta_1}^nx-g_1(x,a^2_{n,v})|&\leq |T_{\beta_1}^nx-g_1(x,y)|+|g_1(x,y)-g_1(x,a^2_{n,v})| \\
 &< \beta_1^{-n\tau_{1}(x)}+L|y-a^2_{n,v}|\\
 &\leq \beta_1^{-n\theta_{1}}+L\beta_2^{-n} \\
 &\leq (L+1)\beta_1^{-n\theta_{1}},
 \end{align*}
 where the last inequality holds  since $\beta_2\geq \beta_1^{\theta_1}$. 
 Similarly, we have  
 \begin{align*}
 |T_{\beta_2}^ny-g_2(a^1_{n,w},y)|&< \beta_2^{-n\theta_{2}}+L\beta_1^{-n}\leq (L+1)\beta_2^{-n\theta_{2}},
 \end{align*}
 where in  the second inequality we have used the assumption that   $\beta_1\geq \beta_2^{\theta_2}$. 

Let 
\[J_{n,1}(w,v)=\left\{x\in I_{n,\beta_1}(w): |T_{\beta_1}^nx-g_1(x,a^2_{n,v})|<(L+1)\beta_1^{-n\theta_{1}}\right\},\]
\[J_{n,2}(w,v)=\left\{y\in I_{n,\beta_2}(v): |T_{\beta_2}^ny-g_2(a^1_{n,w},y)|<(L+1)\beta_2^{-n\theta_{2}}\right\}.\]
Then from the above we see that 
\begin{align}
G&=\bigcap_{N=1}^{\infty}\bigcup_{n=N}^{\infty}\bigcup_{w\in \Sigma_{\beta_1}^n, v\in \Sigma_{\beta_2}^n}\left\{(x,y)\in I_{n,\beta_1}(w)\times I_{n,\beta_2}(v): \begin{aligned}|T_{\beta_1}^nx-g_1(x,y)|<\beta_1^{-n\tau_1(x)}\\
|T_{\beta_2}^ny-g_2(x,y)|<\beta_2^{-n\tau_2(y)}\end{aligned}\right\}\label{eqbetag1g2}\\
&\subset\bigcap_{N=1}^{\infty}\bigcup_{n=N}^{\infty}\bigcup_{w\in \Sigma_{\beta_1}^n, v\in \Sigma_{\beta_2}^n}J_{n,1}(w,v)\times J_{n,2}(w,v).\nonumber
\end{align}
Moreover, by the same argument as in \eqref{eqbeta1nLxx}, we see that for all large $n$, 
\[|J_{n,1}(w,v)|\leq 4(L+1)\beta_1^{-n(1+\theta_1)}, \quad |J_{n,2}(w,v)|\leq 4(L+1)\beta_2^{-n(1+\theta_2)}.\]
This is completely analogous to  \eqref{Jnbeta1} and \eqref{Jnbeta2}. Then  the   upper bounds part of Theorem \ref{thm3} follows from the same proof of  that of Theorem \ref{thm2}.

{\bf The lower bounds part}. Fix full words $w\in \Sigma_{\beta_1}^n$ and $v\in \Sigma_{\beta_2}^n$.  Applying Lemma \ref{lemxnw}(ii) to the function $x\mapsto g_1(x, a^2_{n,v})$ (which is Lipschitz since $g_1:[0,1]^2\to [0,1]$ is Lipschitz), there exists a point $x_{n,w,v}\in I_{n,\beta_1}(w)$ such that 
\[|T_{\beta_1}^nx_{n,w,v}-g_1(x_{n,w,v}, a^2_{n,v})|<\beta_2^{-n}.\] 
Thus for any $(x,y)\in I_{n,\beta_1}(w)\times I_{n,\beta_2}(v)$, we have  by \eqref{eqgiLip} that
\begin{align*}
|T_{\beta_1}^nx-g_1(x,y)|-\beta_2^{-n}&\leq |T_{\beta_1}^nx-g_1(x,y)|-|T_{\beta_1}^nx_{n,w,v}-g_1(x_{n,w,v}, a^2_{n,v})|\\
& \leq |T_{\beta_1}^nx-T_{\beta_1}^nx_{n,w,v}|+|g_1(x,y)-g_1(x_{n,w,v}, a^2_{n,v})|\\
&\leq \beta_1^n|x-x_{n,w,v}|+L(|x-x_{n,w,v}|+|y-a^2_{n,v}|)\\
&\leq (\beta_1^n+L)|x-x_{n,w,v}|+L\beta_2^{-n}.
\end{align*}
Since $\beta_2> \beta_1^{\kappa_1}$ (recall that $\kappa_i=\max_{x\in[0,1]}\tau_i(x)$ for $i=1,2$),   we see that if $|x-x_{n,w,v}|<\frac{1}{4}\beta_1^{-n(1+\tau_1(x))}$, then there exists $n_0\in\N$ (independent of $(x,y)$) such that for all $n\geq n_0$, 
\[|T_{\beta_1}^nx-g_1(x,y)|<2\beta_1^n\cdot \frac{1}{4}\beta_1^{-n(1+\tau_1(x))}+(L+1)\beta_2^{-n}\leq \beta_1^{-n\tau_1(x)}.\]
Similarly, there is a point $y_{n,w,v}\in I_{n,\beta_2}(v)$ such that 
\[|T_{\beta_2}^ny_{n,w,v}-g_2(a^1_{n,w}, y_{n,w,v})|<\beta_1^{-n}.\] 
Thus by a similar argument as above, we have 
\begin{align*}
|T_{\beta_2}^ny-g_2(x,y)|-\beta_1^{-n} \leq (\beta_2^n+L)|y-y_{n,w,v}|+L\beta_1^{-n}.
\end{align*}
Hence  if $|y-y_{n,w,v}|<\frac{1}{4}\beta_2^{-n(1+\tau_2(y))}$, then for large $n$,
\[|T_{\beta_2}^ny-g_2(x,y)|<2\beta_2^n\cdot \frac{1}{4}\beta_2^{-n(1+\tau_2(y))}+(L+1)\beta_1^{-n}\leq \beta_2^{-n\tau_2(y)},\]
where  the second inequality holds for large $n$ since $\beta_1>\beta_2^{\kappa_2}$. 

Let 
\[\tilde{J}_{n,1}(w,v)=\left\{x\in I_{n,\beta_1}(w): |x-x_{n,w,v}|< \frac{1}{4}\beta_1^{-n(1+\tau_1(x))}\right\},\]
\[\tilde{J}_{n,2}(w,v)=\left\{y\in I_{n,\beta_2}(v): |y-y_{n,w,v}|< \frac{1}{4}\beta_2^{-n(1+\tau_2(y))}\right\}.\]
Then from the above (and recall \eqref{eqbetag1g2}), we see that 
\begin{align*}
G\supset \bigcap_{N=1}^{\infty}\bigcup_{n=N}^{\infty}\bigcup_{w\in\Sigma_{\beta_1}^n, v\in \Sigma_{\beta_2}^n\text{ full}}\tilde{J}_{n,1}(w,v)\times \tilde{J}_{n, 2}(w,v).
\end{align*}
This is in complete analogy  to \eqref{eqWsubset}. 
Then a proof similar to the second part of Subsection \ref{subs1} yields the  lower bounds in Theorem \ref{thm3}. 

{\noindent \bf  Acknowledgements}. The author thanks Weiliang Wang for some  comments. He also wish to thank the anonymous referees for their suggestions that led to the improvement of the paper.

\end{document}